\documentclass{amsart}%
\usepackage{amssymb,amsmath,amsfonts,latexsym,amsthm,geometry,graphicx}
\usepackage{amsmath}
\usepackage{amsfonts}
\usepackage{amssymb}
\usepackage{graphicx}%
\providecommand{\U}[1]{\protect\rule{.1in}{.1in}}
\providecommand{\U}[1]{\protect\rule{.1in}{.1in}}
\providecommand{\U}[1]{\protect\rule{.1in}{.1in}}
\newtheorem{theorem}{Theorem}[section]

\newtheorem{proposition}[theorem]{Proposition}

\theoremstyle{definition}

\begin{document}
\title[Lower bounds for the complex polynomial Hardy--Littlewood inequality]{Lower bounds for the complex polynomial Hardy--Littlewood inequality}
\date{}
\author[G. Ara\'{u}jo]{Gustavo Ara\'{u}jo}
\address{Departamento de Matem\'{a}tica \\
Universidade Federal da Para\'{\i}ba \\
58.051-900 - Jo\~{a}o Pessoa, Brazil.}
\email{gdasaraujo@gmail.com}
\author[D. Pellegrino]{Daniel Pellegrino}
\address{Departamento de Matem\'{a}tica \\
Universidade Federal da Para\'{\i}ba \\
58.051-900 - Jo\~{a}o Pessoa, Brazil.}
\email{pellegrino@pq.cnpq.br and dmpellegrino@gmail.com}
\keywords{Absolutely summing operators, Hardy--Littlewod inequality, Bohnenblust--Hille inequality}

\begin{abstract}
The Hardy--Littlewood inequality for complex homogeneous polynomials asserts
that given positive integers $m\geq2$ and $n\geq1$, if $P$ is a complex
homogeneous polynomial of degree $m$ on $\ell_{p}^{n}$ with $2m\leq
p\leq\infty$ given by $P(x_{1},\ldots,x_{n})=\sum_{|\alpha|=m}a_{\alpha
}\mathbf{{x}^{\alpha}}$, then there exists a constant $C_{\mathbb{C}%
,m,p}^{\mathrm{pol}}\geq1$ (which is does not depend on $n$) such that
\[
\left(  {\sum\limits_{\left\vert \alpha\right\vert =m}}\left\vert a_{\alpha
}\right\vert ^{\frac{2mp}{mp+p-2m}}\right)  ^{\frac{mp+p-2m}{2mp}}\leq
C_{\mathbb{C},m,p}^{\mathrm{pol}}\left\Vert P\right\Vert ,
\]
with $\Vert P\Vert:=\sup_{z\in B_{\ell_{p}^{n}}}|P(z)|$. In this short note,
among other results, we provide nontrivial lower bounds for the constants
$C_{\mathbb{C},m,p}^{\mathrm{pol}}$. For instance we prove that, for $m\geq2$
and $2m\leq p<\infty$,
\[
C_{\mathbb{C},m,p}^{\mathrm{pol}}\geq2^{\frac{m}{p}}%
\]
for $m$ even, and
\[
C_{\mathbb{C},m,p}^{\mathrm{pol}}\geq2^{\frac{m-1}{p}}%
\]
for $m$ odd. Estimates for the case $p=\infty$ (this is the particular case of
the complex polynomial Bohnenblust--Hille inequality) were recently obtained
by D. Nu\~{n}ez-Alarc\'{o}n in 2013.

\end{abstract}
\maketitle

\section{Introduction}

Let $\mathbb{K}$ denote the field of real or complex scalars. Given
$\alpha=(\alpha_{1},\ldots,\alpha_{n})\in{\mathbb{N}}^{n}$, define
$|\alpha|:=\alpha_{1}+\cdots+\alpha_{n}$ and $\mathbf{x}^{\alpha}$ stands for
the monomial $x_{1}^{\alpha_{1}}\cdots x_{n}^{\alpha_{n}}$ for $\mathbf{x}%
=(x_{1},\ldots,x_{n})\in{\mathbb{K}}^{n}$. The polynomial Bohnenblust--Hille
inequality (see \cite{alb, bh} and the references therein) ensures that, given
positive integers $m\geq2$ and $n\geq1$, if $P$ is a homogeneous polynomial of
degree $m$ on $\ell_{\infty}^{n}$ given by $P(x_{1},...,x_{n})=\sum
_{|\alpha|=m}a_{\alpha}\mathbf{{x}^{\alpha}}$, then
\[
\left(  {\sum\limits_{\left\vert \alpha\right\vert =m}}\left\vert a_{\alpha
}\right\vert ^{\frac{2m}{m+1}}\right)  ^{\frac{m+1}{2m}}\leq B_{\mathbb{K}%
,m}^{\mathrm{pol}}\left\Vert P\right\Vert
\]
for some constant $B_{\mathbb{K},m}^{\mathrm{pol}}\geq1$ which does not depend
on $n$ (the exponent $\frac{2m}{m+1}$ is optimal), where $\Vert P\Vert
:=\sup_{z\in B_{\ell_{\infty}^{n}}}|P(z)|$.

The search of precise estimates of the growth of the constants $B_{\mathbb{K}%
,m}^{\mathrm{pol}}$ is fundamental for different applications and remains an
important open problem (see \cite{bps} and the references therein). For real
scalars it was shown in
\cite{camposjimenezrodriguezmunozfernandezpellegrinoseoanesepulveda2014} that
\[
\left(  1.17\right)  ^{m}\leq B_{\mathbb{R},m}^{\mathrm{pol}}\leq
C(\varepsilon)\left(  2+\varepsilon\right)  ^{m},
\]
where $C(\varepsilon)\left(  2+\varepsilon\right)  ^{m}$ means that given
$\varepsilon>0$, there is a constant $C\left(  \varepsilon\right)  >0$ such
that $B_{\mathbb{R},m}^{\mathrm{pol}}\leq C(\varepsilon)\left(  2+\varepsilon
\right)  ^{m}$ for all $m$. In other words, for real scalars the
hypercontractivity of $B_{\mathbb{R},m}^{\mathrm{pol}}$ is optimal. For
complex scalars the behavior of $B_{\mathbb{K},m}^{\mathrm{pol}}$ is still
unknown. The best information we have thus far about $B_{\mathbb{C}%
,m}^{\mathrm{pol}}$ are due D. N\'{u}\~{n}ez-Alarc\'{o}n \cite{nunez} (lower
bounds) and F. Bayart, D. Pellegrino and J.B. Seoane-Sep\'{u}lveda \cite{bps}
(upper bounds)
\[%
\begin{array}
[c]{l}%
B_{\mathbb{C},m}^{\mathrm{pol}}\geq\left\{
\begin{array}
[c]{lcl}%
\displaystyle\left(  1+\frac{1}{2^{m-1}}\right)  ^{\frac{1}{4}} &  & \text{for
}m\text{ even};\vspace{0.2cm}\\
\displaystyle\left(  1+\frac{1}{2^{m-1}}\right)  ^{\frac{m-1}{4m}} &  &
\text{for }m\text{ odd};
\end{array}
\right. \\
B_{\mathbb{C},m}^{\mathrm{pol}}\leq C(\varepsilon)\left(  1+\varepsilon
\right)  ^{m}.
\end{array}
\]

The natural extension to $\ell_{p}$ spaces of the polynomial
Bohnenblust--Hille inequality is called polynomial Hardy--Littlewood
inequality (see \cite{n, hardy, pra} and the references therein). More
precisely, given positive integers $m\geq2$ and $n\geq1$, if $P$ is a
homogeneous polynomial of degree $m$ on $\ell_{p}^{n}$ with $2m\leq
p\leq\infty$ given by $P(x_{1},\ldots,x_{n})=\sum_{|\alpha|=m}a_{\alpha
}\mathbf{{x}^{\alpha}}$, then there exists a constant $C_{\mathbb{K}%
,m,p}^{\mathrm{pol}}\geq1$ (which does not depend on $n$) such that
\[
\left(  {\sum\limits_{\left\vert \alpha\right\vert =m}}\left\vert a_{\alpha
}\right\vert ^{\frac{2mp}{mp+p-2m}}\right)  ^{\frac{mp+p-2m}{2mp}}\leq
C_{\mathbb{K},m,p}^{\mathrm{pol}}\left\Vert P\right\Vert ,
\]
with $\Vert P\Vert:=\sup_{z\in B_{\ell_{p}^{n}}}|P(z)|$. Using the generalized
Kahane--Salem--Zygmund inequality (see, for instance, \cite{alb}) we can
verify that the exponents $\frac{2mp}{mp+p-2m}$ are optimal for $2m\leq
p\leq\infty$. When $p=\infty$, since $\frac{2mp}{mp+p-2m}=\frac{2m}{m+1}$, we
recover the polynomial Bohnenblust--Hille inequality. In a more genreal point
of view this kind of results can be seen as coincidence results of the theory
of absolutely summing operators (see \cite{die}).

Very recently, the authors in collaboration with P. Jim\'{e}nez-Rodriguez,
G.A. Mu\~{n}oz-Fern\'{a}ndez, D. N\'{u}\~{n}ez-Alarc\'{o}n, J.B.
Seoane-Sep\'{u}lveda and D. M. Serrano-Rodr\'{\i}guez (see \cite{ajmnpss})
proved that for real scalars and $m\geq2$, the constants of the polynomial
Hardy--Littlewood inequality has at least an hypercontractive growth. More
precisely, it was proved that, for all positive integers $m\geq2$ and all
$2m\leq p<\infty$,
\[
\left(  \sqrt[16]{2}\right)  ^{m}\leq2^{\frac{m^{2}p+10m-p-6m^{2}-4}{4mp}}\leq
C_{\mathbb{R},m,p}^{\mathrm{pol}}\leq C_{\mathbb{R},m,p}^{\mathrm{mult}}%
\frac{m^{m}}{\left(  m!\right)  ^{\frac{mp+p-2m}{2mp}}},
\]
where $C_{\mathbb{R},m,p}^{\mathrm{mult}}$ are the constants of the real case
of the multilinear Hardy-Littlewood inequality (for estimates of these
constants see \cite{ap, aps2014}).

In the case of complex scalars (and concerning upper bounds) similar results
were proved (see \cite{ajmnpss}):
\[
1\leq C_{\mathbb{C},m,p}^{\mathrm{pol}}\leq C_{\mathbb{C},m,p}^{\mathrm{mult}%
}\frac{m^{m}}{\left(  m!\right)  ^{\frac{mp+p-2m}{2mp}}}.
\]
However, there are no lower bounds for $C_{\mathbb{C},m,p}^{\mathrm{pol}}$
that gives us nontrivial information. In this note we provide nontrivial lower
bounds for the constants of the complex case of the polynomial
Hardy--Littlewood inequality. More precisely we prove that, for $m\geq2$ and
$2m\leq p<\infty$,
\[
C_{\mathbb{C},m,p}^{\mathrm{pol}}\geq\displaystyle2^{\frac{m}{p}}%
\]
for $m$ even, and
\[
C_{\mathbb{C},m,p}^{\mathrm{pol}}\geq2^{\frac{m-1}{p}}%
\]
for $m$ odd. For instance,
\[
\sqrt{2}\leq C_{\mathbb{C},2,4}^{\mathrm{pol}}\leq3.1915.
\]

\section{The result}

Let $m\geq2$ be an even positive integer and let $p\geq2m$. Consider the
$2$--homogeneous polynomials $Q_{2}:\ell_{p}^{2}\to\mathbb{C}$ and
$\widetilde{Q_{2}}:\ell_{\infty}^{2}\rightarrow\mathbb{C}$ both given by
$(z_{1},z_{2})\mapsto z_{1}^{2}-z_{2}^{2}+cz_{1}z_{2}$. We know from
\cite{aronklimek2001,camposjimenezrodriguezmunozfernandezpellegrinoseoanesepulveda2014}
that
\[
\Vert\widetilde{Q_{2}}\Vert=\left(  4+c^{2}\right)  ^{\frac{1}{2}}.
\]

If we follow the lines of \cite{nunez} and we define the $m$--homogeneous
polynomial ${Q_{m}}:\ell_{p}^{m}\rightarrow\mathbb{C}$ by ${Q_{m}}%
(z_{1},...,z_{m})=z_{3}\ldots z_{m} Q_{2}(z_{1},z_{2})$ we obtain
\[
\Vert{Q_{m}}\Vert\leq2^{-\frac{m-2}{p}}\Vert{Q_{2}}\Vert\leq2^{-\frac{m-2}{p}%
}\Vert\widetilde{Q_{2}}\Vert=2^{-\frac{m-2}{p}}\left(  4+c^{2}\right)
^{\frac{1}{2}},
\]
where we use the obviuos inequality
\[
\Vert Q_{2}\Vert\leq\Vert\widetilde{Q_{2}}\Vert.
\]
Therefore, for $m\geq2$ even and $c\in\mathbb{R}$, from the polynomial
Hardy--Littlewood inequality it follows that
\[
C_{\mathbb{C},m,p}^{\mathrm{pol}}\geq\frac{\left(  2+|c|^{\frac{2mp}{mp+p-2m}%
}\right)  ^{\frac{mp+p-2m}{2mp}}}{2^{-\frac{m-2}{p}}\left(  4+c^{2}\right)
^{\frac{1}{2}}}.
\]

If
\[
c>\left(  \frac{2^{\frac{2p+4-2m}{p}}-2^{\frac{mp+p-2m}{mp}}}{1-2^{-\frac
{2m-4}{p}}}\right)  ^{\frac{1}{2}},
\]
it is not to difficult to prove that
\[
2^{-\frac{m-2}{p}}\left(  4+c^{2}\right)  ^{\frac{1}{2}}<\left(  \left(
2^{\frac{mp+p-2m}{2mp}}\right)  ^{2}+c^{2}\right)  ^{\frac{1}{2}},
\]
i.e.,%
\[
2^{-\frac{m-2}{p}}\left(  4+c^{2}\right)  ^{\frac{1}{2}}<\left\Vert \left(
2^{\frac{mp+p-2m}{2mp}},c\right)  \right\Vert _{2}.
\]
Since $\frac{2mp}{mp+p-2m}\leq2$, we know that $\ell_{\frac{2mp}{mp+p-2m}%
}\subset\ell_{2}$ and $\Vert\cdot\Vert_{2}\leq\Vert\cdot\Vert_{\frac
{2mp}{mp+p-2m}}$. Therefore, for all
\[
c>\left(  \frac{2^{\frac{2p+4-2m}{p}}-2^{\frac{mp+p-2m}{mp}}}{1-2^{-\frac
{2m-4}{p}}}\right)  ^{\frac{1}{2}},
\]
we have
\begin{align*}
2^{-\frac{m-2}{p}}\left(  4+c^{2}\right)  ^{\frac{1}{2}}  & <\left\Vert
\left(  2^{\frac{mp+p-2m}{2mp}},c\right)  \right\Vert _{2}\\
& \leq\left\Vert \left(  2^{\frac{mp+p-2m}{2mp}},c\right)  \right\Vert
_{\frac{2mp}{mp+p-2m}}\\
& =\left(  2+c^{\frac{2mp}{mp+p-2m}}\right)  ^{\frac{mp+p-2m}{2mp}},
\end{align*}
from which we conclude that
\[
C_{\mathbb{C},m,p}^{\mathrm{pol}}\geq\frac{\left(  2+c^{\frac{2mp}{mp+p-2m}%
}\right)  ^{\frac{mp+p-2m}{2mp}}}{2^{-\frac{m-2}{p}}\left(  4+c^{2}\right)
^{\frac{1}{2}}}>1.
\]

If $m\geq3$ is odd, since $\Vert Q_{m}\Vert\leq\Vert Q_{m-1}\Vert$, then we
have $\Vert Q_{m}\Vert\leq2^{-\frac{m-3}{p}}\left(  4+c^{2}\right)  ^{\frac
{1}{2}}$ and thus we can now proceed analogously to the even case and finally
conclude that for
\[
c>\left(  \frac{2^{\frac{2p+6^{-}2m}{p}}-2^{\frac{mp+p-2m}{mp}}}%
{1-2^{-\frac{2m-6}{p}}}\right)  ^{\frac{1}{2}}%
\]
we have
\[
C_{\mathbb{C},m,p}^{\mathrm{pol}}\geq\frac{\left(  2+c^{\frac{2mp}{mp+p-2m}%
}\right)  ^{\frac{mp+p-2m}{2mp}}}{2^{-\frac{m-3}{p}}\left(  4+c^{2}\right)
^{\frac{1}{2}}}>1.
\]

So we have:

\begin{theorem}
\label{777}Let $m\geq2$ be a positive integer and let $p\geq2m$. Then, for
every $\epsilon>0$,
\[
C_{\mathbb{C},m,p}^{\mathrm{pol}}\geq\frac{\left(  2+\left(  \left(
\frac{2^{\frac{2p+4-2m}{p}}-2^{\frac{mp+p-2m}{mp}}}{1-2^{-\frac{2m-4}{p}}%
}\right)  ^{\frac{1}{2}}+\epsilon\right)  ^{\frac{2mp}{mp+p-2m}}\right)
^{\frac{mp+p-2m}{2mp}}}{2^{-\frac{m-2}{p}}\left(  4+\left(  \left(
\frac{2^{\frac{2p+4-2m}{p}}-2^{\frac{mp+p-2m}{mp}}}{1-2^{-\frac{2m-4}{p}}%
}\right)  ^{\frac{1}{2}}+\epsilon\right)  ^{2}\right)  ^{\frac{1}{2}}%
}>1\ \ \ \text{ if }m\text{ is even}%
\]
and
\[
C_{\mathbb{C},m,p}^{\mathrm{pol}}\geq\frac{\left(  2+\left(  \left(
\frac{2^{\frac{2p+6^{-}2m}{p}}-2^{\frac{mp+p-2m}{mp}}}{1-2^{-\frac{2m-6}{p}}%
}\right)  ^{\frac{1}{2}}+\epsilon\right)  ^{\frac{2mp}{mp+p-2m}}\right)
^{\frac{mp+p-2m}{2mp}}}{2^{-\frac{m-3}{p}}\left(  4+\left(  \left(
\frac{2^{\frac{2p+6^{-}2m}{p}}-2^{\frac{mp+p-2m}{mp}}}{1-2^{-\frac{2m-6}{p}}%
}\right)  ^{\frac{1}{2}}+\epsilon\right)  ^{2}\right)  ^{\frac{1}{2}}%
}>1\ \ \ \text{ if }m\text{ is odd}.
\]

\end{theorem}

Howerver we have another approach to the problem, which is surprisingly
simpler than the above approach and still seems to give best (bigger) lower
bounds for the constants of the polynomial Hardy--Littlewood inequality.

\begin{theorem}
\label{main} Let $m\geq2$ be a positive integer and let $p\geq2m$. Then
\[
C_{\mathbb{C},m,p}^{\mathrm{pol}} \geq\left\{
\begin{array}
[c]{lcl}%
\displaystyle 2^{\frac{m}{p}} &  & \text{for } m \text{ even}; \vspace
{0.2cm}\\
\displaystyle 2^{\frac{m-1}{p}} &  & \text{for } m \text{ odd};
\end{array}
\right.
\]

\end{theorem}

\begin{proof}
Consider $P_{2}:\ell_{p}^{2}\rightarrow\mathbb{C}$ the $2$--homogeneous
polynomial given by $\mathbf{z}\mapsto z_{1}z_{2}$. Observe that
\[
\Vert P_{2}\Vert=\sup_{|z_{1}|^{p}+|z_{2}|^{p}=1}|z_{1}z_{2}|=\sup_{|z|\leq
1}\left\vert z\right\vert (1-\left\vert z\right\vert ^{p})^{\frac{1}{p}%
}=2^{-\frac{2}{p}}.
\]

More generally, if $m\geq2$ is even and $P_{m}$ is the $m$--homogeneous
polynomial given by $\mathbf{z}\mapsto z_{1}\cdots z_{m}$, then
\[
\Vert P_{m}\Vert\leq2^{-\frac{m}{p}}.
\]
Therefore, from the polynomial Hardy--Littlewood inequality we know that
\[
C_{\mathbb{C},m,p}^{\mathrm{pol}}\geq\frac{\left(  \displaystyle\sum
_{|\alpha|=m}|a_{\alpha}|^{\frac{2mp}{mp+p-2m}}\right)  ^{\frac{mp+p-2m}{2mp}%
}}{\Vert P_{m}\Vert}\geq\frac{1}{2^{-\frac{m}{p}}}=2^{\frac{m}{p}}.
\]

If $m\geq3$ is odd, we define again the $m$--homogeneous polynomial $P_{m}$
given by $\mathbf{z}\mapsto z_{1}\cdots z_{m}$ and since $\Vert P_{m}\Vert
\leq\Vert P_{m-1}\Vert$, then we have $\Vert P_{m}\Vert\leq2^{-\frac{m-1}{p}}$
and thus
\[
C_{\mathbb{C},m,p}^{\mathrm{pol}}\geq\frac{1}{2^{-\frac{m-1}{p}}}%
=2^{\frac{m-1}{p}}.
\]

\end{proof}

\section{Comparing the estimates}

The estimates of Theorem \ref{777} seems to become better when $\epsilon$
grows (this seems to be a clear sign that we should avoid the terms $z_{1}%
^{2}$ and $z_{2}^{2}$ in our approach). Making $\epsilon\rightarrow\infty$ in
Theorem \ref{777} we obtain%
\[
C_{\mathbb{C},m,p}^{\mathrm{pol}}\geq\left\{
\begin{array}
[c]{lcl}%
\displaystyle2^{\frac{m-2}{p}} &  & \text{for }m\text{ even};\vspace{0.2cm}\\
\displaystyle2^{\frac{m-3}{p}} &  & \text{for }m\text{ odd},
\end{array}
\right.
\]
which are slightly worse than the estimates from Theorem \ref{main}.

\section{The case $m<p<2m$}

For the case $m<p<2m$, there is also a version of the polynomial
Hardy--Littlewood inequalities (see \cite{dimant}): there exists a constant
$C_{\mathbb{K},m,p}^{\mathrm{pol}}\geq1$ such that, for all positive integers
$n$ and all continuous $m$--homogeneous polynomial $P:\ell_{p}\rightarrow
\mathbb{K}$ given by $P(x_{1},...,x_{n})=\sum_{|\alpha|=m}a_{\alpha
}\mathbf{{x}^{\alpha}}$ we have%
\begin{equation}
\left(  \sum_{\left\vert \alpha\right\vert =m}^{n}\left\vert a_{\alpha
}\right\vert ^{\frac{p}{p-m}}\right)  ^{\frac{p-m}{p}}\leq C_{\mathbb{K}%
,m,p}^{\mathrm{pol}}\left\Vert T\right\Vert \label{ohl}%
\end{equation}
and the exponent $\frac{p}{p-m}$ is optimal. Using a polarization argument
(as, for instance in
\cite{camposjimenezrodriguezmunozfernandezpellegrinoseoanesepulveda2014}), but
this procedure is essencially folklore, we have:

\begin{proposition}
\label{pro:first_approach} If $P$ is a homogeneous polynomial of degree $m$ on
$\ell_{p}^{n}$ with $m<p<2m$ given by $P(x_{1},\ldots,x_{n})=\sum_{|\alpha
|=m}a_{\alpha}\mathbf{{x}^{\alpha}}$, then
\[
\left(  {\sum\limits_{\left\vert \alpha\right\vert =m}}\left\vert a_{\alpha
}\right\vert ^{\frac{p}{p-m}}\right)  ^{\frac{p-2}{p}}\leq C_{\mathbb{K}%
,m,p}^{\mathrm{pol}}\left\Vert P\right\Vert
\]
with
\[
C_{\mathbb{K},m,p}^{\mathrm{pol}}\leq C_{\mathbb{K},m,p}^{\mathrm{mult}}%
\frac{m^{m}}{\left(  m!\right)  ^{\frac{p-m}{p}}},
\]
where $C_{\mathbb{K},m,p}^{\mathrm{mult}}$ are the constants of the
multilinear Hardy-Littlewood inequality.
\end{proposition}

With the same argument used in the proof of Theorem \ref{777} we obtain
similar estimates for the case $m<p<2m,$ i.e.,
\[
C_{\mathbb{C},m,p}^{\mathrm{pol}}\geq\left\{
\begin{array}
[c]{lcl}%
\displaystyle2^{\frac{m}{p}} &  & \text{for }m\text{ even};\vspace{0.2cm}\\
\displaystyle2^{\frac{m-1}{p}} &  & \text{for }m\text{ odd.}%
\end{array}
\right.
\]

\

\end{document}